\documentclass{article}
\usepackage[utf8]{inputenc}
\usepackage{fullpage,tikz,mathrsfs, amsmath, amsthm, amssymb,placeins,mathtools}

\newtheorem{thm}{Theorem}[section]
\newtheorem{lemma}[thm]{Lemma}
\newtheorem{cor}[thm]{Corollary}
\newtheorem{quest}[thm]{Question}
\newtheorem{prop}[thm]{Proposition}

    \makeatletter
\def\@fnsymbol#1{\ensuremath{\ifcase#1\or *\or \ddagger\or
   \mathsection\or \mathparagraph\or \|\or **\or \dagger\dagger
   \or \ddagger\ddagger \else\@ctrerr\fi}}
    \makeatother

\title{Lower bounds for rainbow Tur\'{a}n numbers of paths and other trees}

\author{
Daniel Johnston\thanks{Department of Mathematics, Grand Valley State University, Allendale, Michigan 49401, USA.}
\and
Puck Rombach\thanks{Department of Mathematics \& Statistics, University of Vermont, Burlington, Vermont 05405, USA.}
}

\date{\today}

\begin{document}

\maketitle

\abstract{For a fixed graph $F$, we would like to determine the maximum number of edges in a properly edge-colored graph on $n$ vertices which does not contain a rainbow copy of $F$, that is, a copy of $F$ all of whose edges receive a different color. This maximum, denoted by $ex^*(n, F)$, is the \emph{rainbow Tur\'{a}n number} of $F$. We show that $ex^*(n,P_k)\geq \frac{k}{2}n  + O(1)$ where $P_k$ is a path on $k\geq 3$ edges, generalizing a result by Maamoun and Meyniel and by Johnston, Palmer and Sarkar. We show similar bounds for brooms on $2^s-1$ edges and diameter $\leq 10$ and a few other caterpillars of small diameter.}

\section{Introduction}
Keevash, Mubayi, Sudakov, and Verstra\"ete introduced rainbow Tur\'{a}n numbers in~\cite{keevash2007rainbow}, motivated by a direct application in additive number theory~\cite{ruzsa1978triple}, as well as a desire to study a natural meeting point of Tur\'{a}n and Ramsey type problems, along the lines of~\cite{alon2003properly}. The latter paper describes the problem of finding a rainbow copy of a graph $F$ in a colouring of $K_n$ in which each
colour appears at most $m$ times at every vertex. According to~\cite{keevash2007rainbow}, the rainbow Tur\'{a}n problem is a natural Tur\'{a}n-type extension.
For a fixed graph $F$, the Tur\'{a}n number of $F$, denoted $ex(n,F)$, is the maximum number of edges in a graph on $n$ vertices that contains no copy of $F$. The rainbow Tur\'{a}n number of $F$, denoted $ex^* (n, F)$, is the maximum number of edges in a properly edge-colored graph on $n$ vertices that contains no rainbow copy of $F$. That is, a copy of $F$ whose edges all receive a different color. In~\cite{keevash2007rainbow}, the authors showed that, when a $F$ is not bipartite,
$$ex^*(n,F)=(1+o(1))ex(n,F).$$
Many open questions remain for bipartite graphs. In~\cite{keevash2007rainbow}, the authors showed that, when a $F$ is bipartite, 
$$ex^*(n,K_{s,t})=O(n^{2-\frac{1}{s}}),$$
where $K_{s,t}$ is the complete bipartite graph with partition classes of size $s$ and $t$ such that $s \leq t$. For even cycles, the authors prove a lower bound of $$ex^*(n,C_{2k}) = \Omega(n^{1+\frac{1}{k}} )$$ and find a matching upper bound in the case of $k=3$. Das, Lee and Sudakov~\cite{das2013rainbow} showed that for every fixed integer $k \geq 2$, $$ex^*(n,C_{2k}) = O\left( n^{1+\frac{(1+\epsilon_k) \ln k}{k}} \right),$$ where $\epsilon_k \to 0$ as $k \to \infty$.

In~\cite{johnston2016rainbow}, Johnston, Palmer and Sarkar showed that when $F$ is a forest of $k$ stars, $ex^*(n,F)$ is the maximum value of $(k-1)n+O(1)$ or $\frac{1}{2}(|e(F)| -1)n+O(1)$. They also showed that $ex^*(n,P_k)=\frac{k}{2}n+O(1)$ for $k\in \{ 3,4\}$. Here, we generalize this result to all values $k\geq 3$. In~\cite{johnston2016rainbow}, the authors also showed an upper bound of $ex^*(n,P_k)\leq \lceil \frac{3k-2}{2}n \rceil$. This was improved to $$ex^*(n,P_k) < \left( \frac{9k-5}{7}\right)n$$ by Ergemlidze,  Gy{\H{o}}ri and Methuku~\cite{ergemlidze2018rainbow}, and this is currently the best known upper bound.

In~\cite{alon2016many}, Alon and Shikhelman introduced the following generalized Tur\'{a}n problem: for fixed graphs $H$ and $F$, what is the maximum number of copies of $H$, denoted by $ex(n,H,F)$, that can appear in an $n$-vertex $F$-free graph? The special case $ex(n,C_3,C_5)$ was studied earlier in~\cite{bollobas2008pentagons}. This problem has applications in query complexity of testing graph properties~\cite{gishboliner2018generalized}. This problem extends naturally to a rainbow Tur\'{a}n version, which is suggested in~\cite{halfpap2018}.

The rest of this paper is organized as follows. In Section~\ref{sec:defs} we give a few basic definitions, notation, and facts that will be used throughout the paper. In particular, we describe the two constructions that are the basis for the new lower bounds on $ex^*(n,F)$ for several bipartite graphs $F$. In Section~\ref{sec:paths}, we give new lower bounds on $ex^*(n,P_k)$. Section~\ref{sec:brooms}, we give new lower bounds, and upper bounds, on $ex^*(n,F)$ for some broom graphs, other caterpillars and a few other small trees. Finally, in Section~\ref{sec:open}, we list a few of the many open question that remain.

\section{Definitions, notation and basic results}\label{sec:defs}
Let $G=(V,E)$ be a \emph{graph} on vertex set $V$ and edge set $E \subseteq \binom{V}{2}$. For a vertex $v \in V(G)$ let $\Gamma_G(v)=\{ w \in V(G) | \{v,w\} \in E(G)\}$ be the neighborhood of $v$ and $d(v)=|\Gamma (v)|$ the degree of $v$. We let $d(G)=\frac{1}{n}\sum_V d(v)$ be the average degree of $G$. We will use the following fact about average vertex degrees.
\begin{prop}\label{prop:avdeg}
If $d(v)<\frac{d(G)}{2}$ for some $v \in V(G)$, then $d(G-v)>d(G)$.
\end{prop}
An edge-colored graph $G^*=(V,E,c)$ is a graph with an \emph{edge coloring} $c:E\to \mathbb{N}$. We will only consider \emph{proper} edge colorings, \emph{i.e.} colorings such that $c(e)\neq c(f)$ if $e \cap f \neq \emptyset$. Many of the lower-bound proofs in the remainder of this paper are based on two extremal edge-colored graphs: $K_{2^s}^*$ and $D_{2^s}^*$. The edge-colored graph $K_{2^s}^*$ is the complete graph on $2^s$ vertices, identified with the vectors in $\mathbb{F}_2^s$. The edge-coloring $c:E(K_{2^s})\to \mathbb{F}_2^s$ is given by $c(vw)=v-w$. The graph $D_{2^s}^*$ is a spanning edge-colored subgraph of $K_{2^s}^*$. An edge $vw$ with color $c(vw)$ is in $D_{2^s}^*$ if and only if $d_H(v,w)\in \{ 1,s\}$, where $d_H(v,w)$ is the Hamming distance between binary vectors $v$ and $w$. Note that $K_{2^s}^*$ is $(2^s-1)$-regular and $D_{2^s}^*$ is $(s+1)$-regular. The latter can be thought of as hypercubes with added ``diagonals". We show examples of $K_{2^2}^* \sim D_{2^2}^*$ and $D_{2^3}^*$ in Figure~\ref{fig:d2d3}.
\begin{figure}[h]
\begin{center}
\begin{tikzpicture}[scale=0.7]]
\draw[black!100,line width=1.5pt,dashed] (0,0) -- (0,3);
\draw[black!100,line width=2pt] (0,0) -- (3,0);
\draw[black!100,line width=1.5pt,dashed] (3,0) -- (3,3);
\draw[black!100,line width=2pt] (0,3) -- (3,3);
\draw[black!30,line width=3pt] (0,0) -- (3,3);
\draw[black!30,line width=3pt] (0,3) -- (3,0);
\draw[fill=black!100,black!100] (0,0) circle (.15);
\draw[fill=black!100,black!100] (3,0) circle (.15);
\draw[fill=black!100,black!100] (0,3) circle (.15);
\draw[fill=black!100,black!100] (3,3) circle (.15);
\draw (1.5,3.5) node{$K_{2^2}^* \sim D_{2^2}^*$};
\begin{scope}[shift={(6,-.5)}]
\draw[black!30,line width=3pt] (0,0) -- (1,1);
\draw[black!30,line width=3pt] (0,3) -- (1,4);
\draw[black!30,line width=3pt] (3,0) -- (4,1);
\draw[black!30,line width=3pt] (3,3) -- (4,4);
\draw[black!30,line width=3pt,dotted] (0,0) to[out=20,in=-120] (4,4);
\draw[black!30,line width=3pt,dotted] (1,1) to[out=60,in=200] (3,3);
\draw[black!30,line width=3pt,dotted] (3,0) to[out=130,in=-70] (1,4);
\draw[black!30,line width=3pt,dotted] (4,1) to[out=130,in=-20] (0,3);
\draw[black!100,line width=1.5pt,dashed] (0,0) -- (0,3);
\draw[black!100,line width=2pt] (0,0) -- (3,0);
\draw[black!100,line width=1.5pt,dashed] (3,0) -- (3,3);
\draw[black!100,line width=2pt] (0,3) -- (3,3);
\draw[fill=black!100,black!100] (0,0) circle (.15);
\draw[fill=black!100,black!100] (3,0) circle (.15);
\draw[fill=black!100,black!100] (0,3) circle (.15);
\draw[fill=black!100,black!100] (3,3) circle (.15);
\begin{scope}[shift={(1,1)}]
\draw[black!100,line width=1.5pt,dashed] (0,0) -- (0,3);
\draw[black!100,line width=2pt] (0,0) -- (3,0);
\draw[black!100,line width=1.5pt,dashed] (3,0) -- (3,3);
\draw[black!100,line width=2pt] (0,3) -- (3,3);
\draw[fill=black!100,black!100] (0,0) circle (.15);
\draw[fill=black!100,black!100] (3,0) circle (.15);
\draw[fill=black!100,black!100] (0,3) circle (.15);
\draw[fill=black!100,black!100] (3,3) circle (.15);
\end{scope}
\draw (2,4.5) node{$D_{2^3}^*$};
  \end{scope}
\end{tikzpicture}
\caption{Examples of edge colored-graphs $K_{2^2}^*\sim D_{2^2}^*$ and $D_{2^3}^*$.}\label{fig:d2d3}
\end{center}
\end{figure}
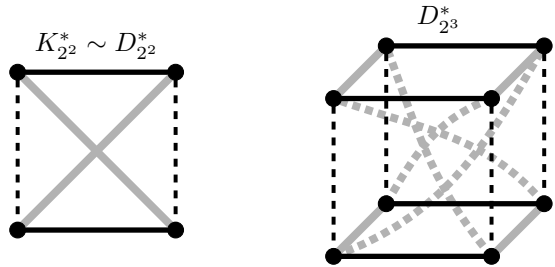
We let $P_k$ be the \emph{path} on $k$ edges (and $k+1$ vertices), and $C_k$ the \emph{cycle} on $k$ edges (and $k$ vertices). The \emph{girth} $g(G)$ of a graph is the minimum $k$ such that $C_k$ is a subgraph of $G$. We define the \emph{broom} $B_{k,l}$ as a tree on $k$ edges that consists of a union of a $P_{l-1}$ and a $K_{1,k-l}$, with an edge between an endpoint of the path and the centre of the star. We let $CP_{(s_1,s_2,\ldots,s_t)}$ be a \emph{caterpillar} that consists of a \emph{central path} $P_{t-1}$ with $s_i$ leaves added to the $i$th vertex on the central path. A broom is a special case of a caterpillar. We show examples of a broom and a caterpillar in Figure~\ref{fig:catbroom}.  
\begin{figure}[h]
\begin{center}
\begin{tikzpicture}[scale=0.7]]
\draw (2,1) node{$B_{10,4}^*$};
\draw[black!100,line width=1.5pt] (0,0) -- (4,0);
\draw[black!100,line width=1.5pt] (0,0) -- (360/7:1);
\draw[black!100,line width=1.5pt] (0,0) -- (2*360/7:1);
\draw[black!100,line width=1.5pt] (0,0) -- (3*360/7:1);
\draw[black!100,line width=1.5pt] (0,0) -- (4*360/7:1);
\draw[black!100,line width=1.5pt] (0,0) -- (5*360/7:1);
\draw[black!100,line width=1.5pt] (0,0) -- (6*360/7:1);
\draw[fill=black!100,black!100] (0,0) circle (.15);
\draw[fill=black!100,black!100] (1,0) circle (.15);
\draw[fill=black!100,black!100] (2,0) circle (.15);
\draw[fill=black!100,black!100] (3,0) circle (.15);
\draw[fill=black!100,black!100] (4,0) circle (.15);
\draw[fill=black!100,black!100] (360/7:1) circle (.15);
\draw[fill=black!100,black!100] (6*360/7:1) circle (.15);
\draw[fill=black!100,black!100] (2*360/7:1) circle (.15);
\draw[fill=black!100,black!100] (3*360/7:1) circle (.15);
\draw[fill=black!100,black!100] (4*360/7:1) circle (.15);
\draw[fill=black!100,black!100] (5*360/7:1) circle (.15);
\begin{scope}[xshift=6cm,yshift=.5cm]
\draw (1,.7) node{$CP_{(3,1,2)}*$};
\draw[black!100,line width=1.5pt] (0,0) -- (2,0);
\draw[fill=black!100,black!100] (0,0) circle (.15);
\draw[fill=black!100,black!100] (1,0) circle (.15);
\draw[fill=black!100,black!100] (2,0) circle (.15);
\draw[black!100,line width=1.5pt] (0,0) -- (-90:1);
\draw[black!100,line width=1.5pt] (0,0) -- (-70:1);
\draw[black!100,line width=1.5pt] (0,0) -- (-110:1);
\draw[black!100,line width=1.5pt] (1,0) -- (1,-.92);
\draw[black!100,line width=1.5pt] (2,0) -- (1.8,-.9);
\draw[black!100,line width=1.5pt] (2,0) -- (2.2,-.9);
\draw[fill=black!100,black!100] (-90:1) circle (.15);
\draw[fill=black!100,black!100] (-70:1) circle (.15);
\draw[fill=black!100,black!100] (-110:1) circle (.15);
\draw[fill=black!100,black!100] (1,-.92) circle (.15);
\draw[fill=black!100,black!100] (1.8,-.9) circle (.15);
\draw[fill=black!100,black!100] (2.2,-.9) circle (.15);
\end{scope}
\end{tikzpicture}
\caption{Example of a broom $B_{10,4}$ and a caterpillar $CP_{(3,1,2)}$.}\label{fig:catbroom}
\end{center}
\end{figure}
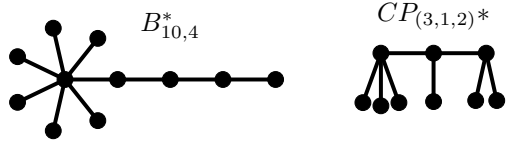
\FloatBarrier
An edge coloring is \emph{rainbow} if the function is injective. We will use the following fact about rainbow paths in an edge-colored graph.
\begin{prop}\label{prop:pathdeg}
If $v \in V(G)$ is the endpoint of a maximal rainbow path $P$ of length $k$ in an edge-colored graph $G$, then $d(v)\leq 2k-2$.
\end{prop}
\begin{proof}
This is true because if an edge that is incident to $v$ cannot be added to $P$ to create a longer rainbow path, then this edge either has a color that already appears on the path (including the edge on $P$ incident to $v$), or the other endpoint of the edge is already on the path, creating a cycle. There can be at most $(l)+(l-2)=2l-2$ such edges.
\end{proof}
In this paper we are predominantly interested in the behavior of $ex^*(n,F)$ as $n \to \infty$. A graph $G$ is \emph{balanced} if $d(H)\leq d(G)$ for all subgraphs $H$ of $G$. The following proposition implies that we need only consider balanced graphs as lower-bound constructions to (rainbow) Tur\'{a}n numbers.
\begin{prop}\label{prop:balanced}
Suppose that $G$ is an edge colored graph with no rainbow copy of some graph $F$, and that $$ex^*(n,F)=\frac{d(G)}{2}n+O(1).$$
Then, $G$ is balanced.
\end{prop}
\begin{proof}
Suppose that $G$ has a subgraph $H$ such that $d(H)>d(G)$. Then, we can construct rainbow $F$-free graphs on $n$ vertices with average degree $d(H)+O(1)$ by taking disjoint copies of $H$ (and a few isolated vertices). This implies $ex^*(n,F)\geq \frac{d(H)}{2}n+O(1)$; a contradiction.
\end{proof}

\section{Lower bound for $P_k$}\label{sec:paths}
In~\cite{maamoun1984problem}, Maamoun and Meyniel showed that $ex^*(n,P_k) \geq \frac{k}{2}n  + O(1)$, when $k+1=2^s$ for some $s \in \mathbb{N}$. We show that this is true for any $k\geq 3$. In ~\cite{keevash2007rainbow}, Keevash, Mubayi, Sudakov and Verstra\"{e}te conjectured that the extremal example for avoiding rainbow $P_k$s is a disjoint union of cliques of size $c(k)$, where $c(k)$ is chosen as large as possible so that $K_{c(k)}$ can be properly edge-colored with no rainbow $P_k$. This conjecture was proven false in ~\cite{johnston2016rainbow}, by providing a non-complete 4-regular edge-colored graph that does not have a $P_4$ and showing that any proper edge-coloring of $K_5$ yields a rainbow copy of $P_4$. This construction is $D_{2^3}^*$ as defined in the previous section. Hence, we generalize the construction to give a properly edge-colored $k$-regular graph that does not have a $P_k$ for any $k\geq 2$. This construction is not the complete graph when $k>3$. 
.

\begin{thm}\label{thm:dm}
	Let $P_k$ be the path of length $k$, then
	\[
	ex^*(n,P_k) \geq \frac{k}{2}n  + O(1).
	\]
\end{thm}

\begin{proof}

Consider the edge-colored graph $D_{2^s}^*$. Suppose that $P$ is a rainbow path of length $k=s+1$ in $D_{2^s}^*$ with endpoints $v$ and $w$. Then, $$v-w=\sum_{e \in E(P)} c(e).$$ However, if $P$ is rainbow, then $c(e(P))=\{ c_1,\ldots, c_{m+1} \} $. This implies that $v-w=0$, which contradicts $P$ being a path.

The graph $D_{2^s}^*$ is $s+1$-regular, and therefore has $\frac{1}{2}n(s+1) = \frac{1}{2}kn$ edges. When $n$ is a multiple of $2^{k-1}$, we can therefore create a rainbow $P_k$-free $k$-regular graph by taking disjoint copies of $D_{2^s}^*$.
\end{proof}

We make an observation here about the edge-colored graph $D_{2^s}^*$ that will be useful in later sections. Let $\{ c_1,c_2,\ldots,c_s \}$ be the standard basis of $\mathbb{F}_2^s$ and let $c_{s+1}$ be the vector of all 1s of length $s$. It is easy to see that $D_{2^s}^*$ does not contain a rainbow cycle of length $<s+1$, by noting that there is no $S \subseteq \{ c_1,\ldots, c_{m+1} \} $ such that $\sum_{S} c =0$. Thus for any graph $F$ with girth $g(F)<k$, we can obtain a properly edge-colored graph containing no rainbow copy of $F$ having $\frac{k}{2}n+O(1)$ edges.  Note that this construction does not improve the lower bound of $ex^* (n,F)$ obtained from known bounds for $ex^* (n,C_k)$.

In~\cite{schrijver2018}, it is shown that, for $k\leq 10$, each properly $k$-edge-colored $k$-regular graph contains a rainbow path of length $k-1$. Theorem~\ref{thm:dm} implies that this result is tight. If it is true that $ex^* (n,P_k)> \frac{k}{2}n$ for $k\leq 10$, then there is no construction similar to $D_{2^s}^*$ that produces extremal graphs: those would be irregular or not $\Delta (G)$-edge-colored.

\section{Caterpillars and other trees}\label{sec:brooms}

We will start this section by focusing on broom graphs, since they are a natural tree to consider between stars and paths.

\begin{lemma}
We have
$$ex^* (n,B_{k,2})=\begin{dcases} \frac{k}{2}n+O(1), \; \mbox{for }k\mbox{ odd,}\\ \frac{k^2}{2(k+1)}n+O(1), \; \mbox{for }k\mbox{ even.} \end{dcases}$$
\end{lemma}

\begin{proof}
If $k$ is odd, then no $K_{k+1}$ with a $k$-edge-coloring contains a rainbow $B_{k,2}$. Suppose that we have a $K_{k+1}$ with a $k$-edge-coloring that contains a rainbow $B_{k,2}$. Let $v_0$ be the vertex of degree $k-1$ in $B_{k,2}$, with edges of colors $1,\ldots,k-1$ incident to $v_0$ in $B_{k,2}$, and let $w$ be the vertex such that $v_0w \notin E(B_{k,2})$. Then $w$ has an edge of color $k$ to a vertex other than $v_0$ in $B_{k,2}$. This is a contradiction, since we must have that edge $v_0w$ has color $k$ in $K_{k+1}$. Therefore, $$ex^* (n,B_{k,2})\geq \frac{k}{2}n+O(1)$$ when $k$ is odd. Let $G$ be a graph with a proper edge coloring, and no rainbow copy of $B_{k,2}$. Suppose that $G$ has a vertex $v_0$ with $d(v_0)\geq k$. If any neighbor of $v_0$ has an edge to a non-neighbor of $v_0$, this gives rise to a copy of $B_{k,2}$. If $d(v_0)>k$, there cannot be any edges in $G[\Gamma (v_0)]$, for the same reason. Therefore, $$ex^* (n,B_{k,2})\leq \frac{k}{2}n.$$
This implies that, when $k$ is odd,
$$ex^* (n,B_{k,2})= \frac{k}{2}n+O(1).$$

If $k$ is even, suppose that $G$ has no $B_{k,2}$ and that $G$ has vertex $v_0$ with $d(v_0)=k$, and edges of colors $1,\ldots,k$ incident to $v_0$. Then there are no other vertices in the component of $v_0$, so we can suppose that $V(G)=v_0 \cup \Gamma (v_0)$. There cannot be an edge of color $>k$ in $G$, as this would give rise to a copy of $B_{k,2}$ in $G$. For every color $1,\ldots,k$, there are at most $(k-2)/2$ edges in $G[\Gamma (v_0)]$, since $|\Gamma (v_0)|=k$ is even and one neighbor of $v_0$ already uses this color on the edge to $v_0$. This implies that $$|E(G)| \leq k+\frac{k(k-2)}{2}=\frac{k^2}{2}=\frac{k^2}{2(k+1)}n.$$ We can construct such a $G$: Take a properly $(k+1)$-edge-colored copy of $K_{k+1}$ and remove all edges of color $k+1$. Now, take edge-disjoint unions of this graph to obtain $$ex^* (n,B_{k,2})=\frac{k^2}{2(k+1)}n+O(1).$$
\end{proof}

\begin{lemma}\label{brupper}
When $l\leq (k+3)/3$, 
$$ex^* (n,B_{k,l})\leq \frac{k+l-2}{2}n.$$
\end{lemma}
\begin{proof} Let $G$ be a graph with mean degree $d(G)> k+l-2$ and let $c$ be a proper edge-coloring of $G$. We suppose that $G$ is balanced, by Proposition~\ref{prop:balanced}. 
We claim that every vertex of $G$ is an endpoint of a rainbow path of length $l$. Suppose that $v$ is a vertex that is not an endpoint of a path of length $l$, and let $u$ be a vertex at the other end of a maximal rainbow path starting at $v$. Then, by Proposition~\ref{prop:pathdeg}, $$d(u)\leq 2l-2<\frac{k+l-2}{2}.$$ However, by Proposition~\ref{prop:avdeg}, this implies that $d(G-u)>d(G)$, contradicting the fact that $G$ is balanced. Therefore, $G$ has a vertex $w$ of degree $d(w)\geq k+l-1 $ that is an endpoint of a rainbow path $P$ of length $l$. Then, $w$ is incident to at most $l$ edges that have colors that occur on the path, and at most $l-1$ further edges that intersect with $P$. This implies that $w$ is incident to at least $k+l-1-(l+l-1)=k-l$ edges that neither intersect $P$ nor have colors in common with $P$. This gives a rainbow copy of $B_{k,l}$.
\end{proof}

\begin{lemma}\label{Bn3}
When $k=2^s-1$ for $3\leq s$, we have
$$ex^* (n,B_{k,3})= \frac{k+1}{2}n+O(1).$$
\end{lemma}

\begin{proof}
Consider the edge-colored graph $K_{2^s}^*$. Suppose that this edge-colored graph contains a rainbow copy of $B_{k,3}$, where $v$ is the center of the star, and $v,x,y,z$ is the broom stick of length $3$, and let $u$ be the vertex not in the copy of $B_{k,3}$. The edges from $v$ of colors $c(xy)$ and $c(yz)$ must go to the set $u,y,z$, and the only possibility is that $c(vu)=c(yz)$ and $c(vz)=c(xy)$. However, due to the definition of $K_{2^s}^*$, $c(vz)=c(xy)$ implies that $c(vx)=c(yz)$, a contradiction. The upper bound is given by Lemma~\ref{brupper}.
\end{proof}

\begin{lemma}
For $4\leq d \leq 9$ and $k=2^s-1$ for some $2\leq s$, we have
$$ex^* (n,B_{k,d})\geq \frac{k}{2}n+O(1).$$
\end{lemma}

\begin{proof}
Consider the edge-colored graph $K_{2^s}^*$. If this edge-colored graph contains a rainbow copy of $B_{k,d}$, this implies that we have a set of distinct vectors $W=\{w_1,w_2,\ldots,w_{d}\}$ (the colors of the edges on the path along the broom stick) such that $\sum_{i=1}^aw_i\in W$ for all $1\leq a \leq d$. Suppose that $K_{2^s}^*$ has a rainbow copy of the broom $B_{k,d}$. For any edge not in the broom that is incident to the center of the star in the broom, and that has another endpoint that is also in the broom, its color must appear somewhere in the rainbow copy of $B_{k,d}$, and this can only be in the broom stick (\emph{i.e.} not incident to the center of the star). It can be verified (by brute force) that such a sequence does not exist for $2\leq d \leq 9$, for vectors of any length. Such a sequence does exist for $d=10$, which shows that $K_{2^s}^*$ contains a rainbow $B_{k,10}$ when $k=2^s-1$ for $s \geq 4$.
\end{proof}

The construction $K_{2^s}^*$ provides lower bounds for a few other caterpillars on $2^s-1$ edges with short central paths, which we list in the following theorem. 
\begin{thm}\label{thm:caterpillars}
Let $F$ be a caterpillar on $k=2^s-1$, $s\geq 2$, edges, and suppose that $F$ is of the form
\begin{itemize}
    \item[(a)] $CP_{(1,t,1)}$, for $t\geq 2$,
    \item[(b)] $CP_{(t,q)}$ for $t,q\geq 2$ odd,
    \item[(c)] $CP_{(t,0,q)}$, for $t,q \geq 2$,
    \item[(d)] $CP_{(t,0,0,q)}$, for $t,q \geq 2$,
    \item[(e)] $CP_{(t,1,q)}$ for $t,q\geq 2$ odd.
\end{itemize}
Then,
$$ex^* (n,F)\geq \frac{k}{2}n +O(1).$$
\end{thm}

\begin{proof} We separate the cases (a), (b), (c,d), (e). For all cases, consider the edge-colored graph $K_{2^s}^*$. 
\begin{itemize}
    \item[(a)] Suppose that this graph has a rainbow copy of $F$. Let $x$ be the center of the star, and let $v$ and $w$ be the vertices at distance 2 from $x$ in $F$, with edges of colors $c_v$ and $c_w$ to vertices $y_v$ and $y_w$, respectively, in $F$. Then $c(xv)=c_w$ and $c(xw)=c_v$. However, this implies that $y_v=y_w$: a contradiction.
\item[(b)] Suppose that $K_{2^s}^*$ has a copy of $F$. Let $x$ and $y$ be the vertices of degree $t$ and $q$, respectively. Then, for all colors other than $c(xy)$, we have a bijection $f(c)=c+c(xy)$, such that pairs of edges in $F$ on colors $c,f(c)$ must both be incident to $x$ or both to $y$ (as we cannot have a path $c,c(xy),f(c)$). This implies that $q$ and $t$ are even.
 \item[(c,d)] In any copy of $CP_{(t,0,q)}$ or $CP_{(t,0,0,q)}$ in this graph, with $x$ and $y$ the endpoints of the central path, no edge can have color $c(xy)$. Therefore, it cannot be rainbow. 
  \item[(e)] In a rainbow copy of $CP_{(t,1,q)}$, let $x,y,z$ be the vertices of the central path. Then the leaf-edge incident to $y$ must have color $c(xz)$, or else this color does not appear in the rainbow copy. The remainder of the argument is similar to the proof of (b). 
\end{itemize}
\end{proof}

\begin{lemma}
Let $F$ be a tree on 7 edges that is not isomorphic to one of the three trees in Figure~\ref{fig:trees}. Then,
$$ex^*(n,F)\geq \frac{7}{2}n + O(1).$$
\end{lemma}
\begin{proof}
By brute force, there is no rainbow copy of $F$ in $K_{2^3}^*$.
\end{proof}

\begin{figure}[h]
\begin{center}
\begin{tikzpicture}[scale=0.7]]
\draw[black!100,line width=1.5pt] (0,0) -- (1,0);
\draw[black!100,line width=1.5pt] (120:1) -- (0,0);
\draw[black!100,line width=1.5pt] (240:1) -- (0,0);
\draw[fill=black!100,black!100] (120:1) circle (.15);
\draw[fill=black!100,black!100] (240:1) circle (.15);
\draw[fill=black!100,black!100] (0,0) circle (.15);
\draw[fill=black!100,black!100] (1,0) circle (.15);
\begin{scope}[xshift=1cm]
\draw[fill=black!100,black!100] (20:1) circle (.15);
\draw[fill=black!100,black!100] (60:1) circle (.15);
\draw[fill=black!100,black!100] (-20:1) circle (.15);
\draw[fill=black!100,black!100] (-60:1) circle (.15);
\draw[black!100,line width=1.5pt] (20:1) -- (0,0);
\draw[black!100,line width=1.5pt] (60:1) -- (0,0);
\draw[black!100,line width=1.5pt] (-20:1) -- (0,0);
\draw[black!100,line width=1.5pt] (-60:1) -- (0,0);
\end{scope}
\begin{scope}[xshift=8cm]
\draw[fill=black!100,black!100] (0,.5) circle (.15);
\draw[fill=black!100,black!100] (1,.5) circle (.15);
\draw[fill=black!100,black!100] (2,.5) circle (.15);
\draw[fill=black!100,black!100] (3,.5) circle (.15);
\draw[fill=black!100,black!100] (4,.5) circle (.15);
\draw[fill=black!100,black!100] (5,.5) circle (.15);
\draw[fill=black!100,black!100] (6,.5) circle (.15);
\draw[fill=black!100,black!100] (4,-.5) circle (.15);
\draw[black!100,line width=1.5pt] (4,.5) -- (4,-.5);
\draw[black!100,line width=1.5pt] (6,.5) -- (0,.5);
\end{scope}
\begin{scope}[xshift=4cm,yshift=.5cm]
\draw[black!100,line width=1.5pt] (0,0) -- (2,0);
\draw[fill=black!100,black!100] (0,0) circle (.15);
\draw[fill=black!100,black!100] (1,0) circle (.15);
\draw[fill=black!100,black!100] (2,0) circle (.15);
\draw[black!100,line width=1.5pt] (0,0) -- (-.2,-.9);
\draw[black!100,line width=1.5pt] (0,0) -- (.2,-.9);
\draw[black!100,line width=1.5pt] (1,0) -- (1,-.92);
\draw[black!100,line width=1.5pt] (2,0) -- (1.8,-.9);
\draw[black!100,line width=1.5pt] (2,0) -- (2.2,-.9);
\draw[fill=black!100,black!100] (-.2,-.9) circle (.15);
\draw[fill=black!100,black!100] (.2,-.9) circle (.15);
\draw[fill=black!100,black!100] (1,-.92) circle (.15);
\draw[fill=black!100,black!100] (1.8,-.9) circle (.15);
\draw[fill=black!100,black!100] (2.2,-.9) circle (.15);
\end{scope}
\end{tikzpicture}
\caption{The only three trees on 7 edges that have rainbow copies in $K_{2^3}^*$.}\label{fig:trees}
\end{center}
\end{figure}
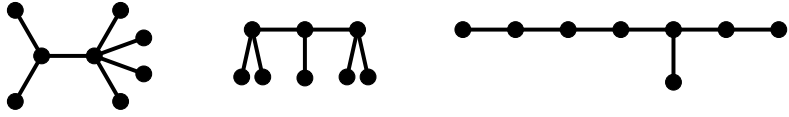

\section{Generalized Tur\'{a}n numbers}\label{sec:generalturan}
Here we consider a rainbow version of the generalized Tur\'{a}n problem suggested in~\cite{alon2016many}. For fixed graphs $H$ and $F$, let the maximum number of rainbow copies of $H$ in a graph with no rainbow copy of $F$ be the generalized Tur\'{a}n number of $H$ and $F$, denoted $ex^*(n,H,F)$. First, consider our graphs that avoid long rainbow paths. In~\cite{halfpap2018}, Halfpap and Palmer use our construction $D_{2^s}^*$ to show that 
$$ex^*(n,C_{k},P_k)\geq \frac{(k-1)!}{2}n+O(1).$$
They also show that $ex^*(n,C_{k},P_k)= \Theta(n).$ We note a few more, similar bounds obtained from this construction in the following corollary. 
\begin{cor}\label{cor:pk-1}
For $k\geq 3$ we have
$$ex^*(n,P_{\ell},P_k)\geq  \frac{k!}{2(k-\ell)!}n+O(1), \;\; \ell\leq k.,$$
and 
$$ex^*(n,P_{\ell},C_k)\geq \frac{(\lfloor \log_2 n \rfloor +1)!}{2(\lfloor \log_2 n \rfloor +1-\ell)!}n+O(1), \;\; \ell\leq k,$$
and 
\begin{align*}ex^*(n,C_{k},\{ C_3,\ldots,C_{k-1}\})&\geq \frac{(k-1)!}{2}n+O(1).
\end{align*}
\end{cor}

\begin{proof}
Consider $D_{2^s}^*$ with $k=s+1$. For any vertex $v$ of $D_{2^s}^*$, and any $x_1,\ldots,x_{\ell}$ of $\ell$ distinct colors from the set $\{c_1,\ldots,c_{k}\}$, there is a unique path in $D_{2^s}^*$ of length $\ell$ that starts at $v$ and whose edges have colors $x_1,\ldots,x_{\ell}$ in order along the path. Since $D_{2^s}^*$ is $k$-regular and properly $k$-edge colored, such a walk must exist, and the structure of $D_{2^s}^*$ prohibits such a walk from intersecting itself. Therefore, correcting for counting each path for both endpoints, this graph contains $\frac{k!}{2(k-\ell)!}n$ rainbow copies of $P_{\ell}$. 

For the second inequality, we count rainbow copies of $P_{\ell}$ in $D_{2^{s}}^*$ for $s\geq k$, which is rainbow $C_{k}$-free. A similar counting argument holds for $C_k$. 

\end{proof}
The third inequality in Corollary~\ref{cor:pk-1} can be restated as follows: the highest number of rainbow copies of $C_k$ in a graph of girth $k$ is at least $n(k-1)!/2+O(1)$.

For the next corollary, we consider the edge-colored graph $K_{2^s}^*$, and note that small cycles are easy to count.

\begin{cor}\label{cor:c3pk}
For $k=2^s-1$, $s\geq 2$, and $F$ a graph on $k$ edges isomorphic to $P_k$ or one of the caterpillars listed in Theorem~\ref{thm:caterpillars}, we have
\begin{align*}ex^*(n,C_3,F)&\geq \frac{k(k-1)}{6}n+O(1),\\
ex^*(n,C_4,F)&\geq \frac{k(k-1)(k-2)}{8}n+O(1),\\
ex^*(n,C_5,F)&\geq \frac{k(k-1)(k-3)(k-7)}{10}n+O(1),\\
ex^*(n,C_\ell,F)&=\Omega(k^{\ell-1}n), \;\; \ell \ll k .\\
\end{align*}
\end{cor}

\section{Open questions}\label{sec:open}
\begin{quest}\label{quest:paths}
In ~\cite{keevash2007rainbow}, Keevash, Mubayi, Sudakov and Verstra\"{e}te conjectured that the extremal example for avoiding rainbow $P_k$s is a disjoint union of cliques. This conjecture was proven false in ~\cite{johnston2016rainbow}, by providing a non-complete 4-regular edge-colored graph that does not have a $P_4$ and showing that any proper edge-coloring of $K_5$ yields a rainbow copy of $P_4$. The generalization of this construction, $D_{2^{k-1}}^*$, given here, is not a complete graph for $k>3$. However, when $k=5$, there is an equivalently dense union of cliques. The geometric construction~\cite{soifer2008mathematical} of a proper edge-coloring of $K_6$, shown in Figure~\ref{fig:k6}, does not have a rainbow copy of $P_5$. (This geometric construction does not work for $K_8$ and avoiding a rainbow $P_7$.) The construction by Maamoun and Meyniel shows that there are proper colorings of $K_n$ that avoid a rainbow $P_{n-1}$ when $n=2^s$ for $s\geq 2$. This leads to two natural questions: does every proper edge coloring of $K_n$ have a rainbow copy of $P_{n-1}$ when $n$ is odd? Is there a proper edge coloring of $K_n$ that avoids a rainbow copy of $P_{n-1}$ for every even $n\geq 4$? In~\cite{alon2017random}, Alon, Pokrovskiy and Sudakov show that every properly edge-colored $K_n$ has a rainbow cycle of length $n-O(n^{3/4})$. This is currently the best we know for general $n$.
\end{quest}

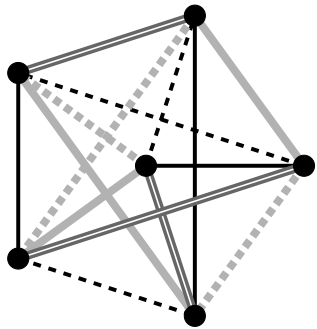
\begin{figure}[h]
\begin{center}
\begin{tikzpicture}[scale=0.7]]
\draw[black!100,line width=1.5pt] (0,0) -- (0:3);
\draw[black!100,line width=1.5pt] (72:3) -- (-72:3);
\draw[black!100,line width=1.5pt] (144:3) -- (-144:3);
\draw[black!100,line width=1.5pt,dashed] (0,0) -- (72:3);
\draw[black!100,line width=1.5pt,dashed] (0:3) -- (144:3);
\draw[black!100,line width=1.5pt,dashed] (-144:3) -- (-72:3);
\draw[black!30,line width=3pt,dotted] (0,0) -- (144:3);
\draw[black!30,line width=3pt,dotted] (72:3) -- (-144:3);
\draw[black!30,line width=3pt,dotted] (0:3) -- (-72:3);
\draw[black!30,line width=3pt] (0,0) -- (-144:3);
\draw[black!30,line width=3pt] (-72:3) -- (144:3);
\draw[black!30,line width=3pt] (0:3) -- (72:3);
\draw[black!60,line width=1.3pt,double] (0,0) -- (-72:3);
\draw[black!60,line width=1.3pt,double] (0:3) -- (-144:3);
\draw[black!60,line width=1.3pt,double] (144:3) -- (72:3);
\draw[fill=black!100,black!100] (0,0) circle (.2);
\draw[fill=black!100,black!100] (72:3) circle (.2);
\draw[fill=black!100,black!100] (144:3) circle (.2);
\draw[fill=black!100,black!100] (216:3) circle (.2);
\draw[fill=black!100,black!100] (288:3) circle (.2);
\draw[fill=black!100,black!100] (0:3) circle (.2);
\end{tikzpicture}
\caption{The geometric proper 5-edge-coloring of $K_6$~\cite{soifer2008mathematical}. This construction avoids a rainbow $P_5$.}\label{fig:k6}
\end{center}
\end{figure}

\begin{quest}
In~\cite{schrijver2018}, it was shown that, for $k\leq 10$, each properly $k$-edge-colored $k$-regular graph contains a rainbow path of length $k-1$. Theorem~\ref{thm:dm} implies that this result is tight, because the construction $D_{2^s}^*$ for $k=s+1$ is a $k$-edge-colored $k$-regular graph with no $P_k$. This question of whether every $k$-edge-colored $k$-regular graph must have a rainbow $P_{k-1}$ is open for $k>10$. A theorem of Babu, Sunil Chandran, and
Rajendraprasad implies that every $k$-edge-colored $k$-regular graph contains a rainbow path of length $\frac{2}{3}k$~\cite{babu2015heterochromatic}.
\end{quest}

\begin{quest}
In ~\cite{pokrovskiy2017edge}, Pokrovskiy and Sudakov define a $t$-spider as a radius 2 tree with $t$ degree 2 vertices (or
equivalently a tree obtained from a star by subdividing $t$ of its edges once), and show that every properly edge-colored $K_n$ contains a (many, in fact) edge-disjoint spanning
rainbow $t$-spider for any $0.0007n \leq t \leq 0.2n$. In Theorem~\ref{thm:caterpillars} we showed that this does not hold for $t=2$. For other values of $t$, must every properly edge-colored $K_n$ have a rainbow $t$-spider?
\end{quest}

\begin{quest}
How many rainbow copies of $C_k$ does $K_{2^s}^*$, for $k=2^s-1$, have? It is easy to see that for large enough $n$, using disjoint copies of $D_{2^s}^*$ is much better than using copies of $K_{2^s}^*$ in terms of maximizing the number of rainbow $C_k$s while avoiding $P_k$. Enumerating rainbow copies of $C_k$ in $K_{2^s}^*$ would tell us more about $ex^*(n,C_k,F)$ when $F$ is another tree, such as one of the caterpillars listed in Theorem~\ref{thm:caterpillars}.
\end{quest}

\begin{quest}
There are still plenty of caterpillars and other trees that are not covered by Theorem~\ref{thm:caterpillars}. Are there other trees that we missed that are not in $K_{2^s}^*$? Are there other subgraphs of $K_{2^s}^*$, along the lines of $D_{2^s}^*$, that efficiently avoid other trees?
\end{quest}

\begin{quest}
Is it true that $ex^*(n,T)\geq ex^*(P_k)$, for any tree $T$ on $k$ edges? Informally, are paths the easiest trees to avoid?
\end{quest}

\FloatBarrier
\bibliographystyle{plain}
\bibliography{rainbow}

\end{document}